\documentclass[12pt,english,reqno]{amsart}

\usepackage{amsthm,amsmath,amssymb,amsfonts}
\usepackage{graphicx}
\usepackage{a4wide}
\usepackage{xcolor}
\usepackage{mathtools}

\usepackage[
  pdfauthor={},
  pdfkeywords={},
  pdftitle={},
  pdfcreator={},
  pdfproducer={},
  linktocpage,colorlinks,bookmarksnumbered]{hyperref}

\newcommand{\myleft}{\mathopen{}\mathclose\bgroup\left}
\newcommand{\myright}{\aftergroup\egroup\right}

\newcommand{\norm}[1]{\left|#1\right|}
\newcommand{\set}[1]{\left\{#1\right\}}
\newcommand{\intpart}[1]{\myleft[#1\myright]}
\newcommand{\fracpart}[1]{\myleft\{#1\myright\}}

\DeclareMathOperator{\cont}{cont}
\DeclareMathOperator{\denom}{denom}
\DeclareMathOperator{\lcm}{lcm}
\DeclareMathOperator{\ord}{ord}
\DeclareMathOperator{\rad}{rad}

\newcommand{\NN}{\mathbb{N}}
\newcommand{\ZZ}{\mathbb{Z}}
\newcommand{\QQ}{\mathbb{Q}}
\newcommand{\RR}{\mathbb{R}}
\newcommand{\PP}{\mathbb{P}}

\newcommand{\pp}{\mathfrak{p}}
\newcommand{\dd}{\mathfrak{d}}

\newcommand{\GP}{\mathcal{P}}
\newcommand{\BT}{\widetilde{B}}
\newcommand{\IT}{\mathcal{I}}
\newcommand{\EO}{\mathcal{E}^{\text{odd}}}
\newcommand{\EE}{\mathcal{E}^{\text{even}}}

\theoremstyle{plain}
\newtheorem{theorem}{Theorem}
\newtheorem{lemma}{Lemma}
\newtheorem{conjecture}{Conjecture}

\theoremstyle{definition}
\newtheorem{Table}{Table}

\newcommand{\arXiv}[1]{\href{http://arxiv.org/abs/#1}{\texttt{arXiv:\,#1~[math.NT]}}}

\begin{document}

\title{On a product of certain primes}
\author{Bernd C. Kellner}
\address{G\"ottingen, Germany}
\email{bk@bernoulli.org}

\begin{abstract}
We study the properties of the product, which runs over the primes,
\[
  \mathfrak{p}_n = \prod_{s_p(n) \, \geq \, p} p
    \quad (n \geq 1),
\]
where $s_p(n)$ denotes the sum of the base-$p$ digits of $n$.
One important property is the fact that $\mathfrak{p}_n$
equals the denominator of the Bernoulli polynomial $B_n(x) - B_n$,
where we provide a short $p$-adic proof. Moreover, we consider the
decomposition $\mathfrak{p}_n = \mathfrak{p}_n^- \cdot \mathfrak{p}_n^+$,
where $\mathfrak{p}_n^+$ contains only those primes $p > \sqrt{n}$.
Let $\omega( \cdot )$ denote the number of prime divisors.
We show that $\omega( \mathfrak{p}_n^+ ) < \sqrt{n}$, while we raise the
explicit conjecture that
\[
  \omega( \mathfrak{p}_n^+ ) \, \sim \, \kappa \, \frac{\sqrt{n}}{\log n}
  \quad \text{as $n \to \infty$}
\]
with a certain constant $\kappa > 1$, supported by several computations.
\end{abstract}

\maketitle


\section{Introduction}

Let $\PP$ be the set of primes.
Throughout this paper, $p$ denotes a prime, and $n$ denotes a nonnegative integer.
The function $s_p(n)$ gives the sum of the base-$p$ digits of $n$.
Let $\GP(n)$ denote the greatest prime factor of $n \geq 2$, otherwise $\GP(n)=1$.
An empty product is defined to be $1$.

We study the product of certain primes,
\begin{equation} \label{eq:primes-prod}
  \pp_n := \prod_{\substack{
    p \, \in \, \PP \\
    s_p(n) \, \geq \, p}} p
    \quad (n \geq 1),
\end{equation}
which is restricted by the condition $s_p(n) \geq p$ on each prime factor $p$.
Since $s_p(n) = n$ in case $p > n$, the product~\eqref{eq:primes-prod} is always
finite.

The values $\pp_n$ are of basic interest, as we will see in Section~\ref{sec:bern-poly},
since they are intimately connected with the denominators of the Bernoulli polynomials
and related polynomials.

Theorem~\ref{thm:primes-bound} below supplies sharper bounds on the prime factors
of $\pp_n$. For the next theorem, giving properties of divisibility,
we need to define the squarefree kernel of an integer as follows:
\[
  \rad(n) := \prod_{p \, \mid \, n} p, \quad
  \rad^*(n) := \left\{
    \begin{array}{ll}
      1,       & \text{if $n$ is prime}, \\
      \rad(n), & \text{else}
    \end{array}
  \right.
  \quad (n \geq 1).
\]

\begin{theorem} \label{thm:primes-div}
The sequence $(\pp_n)_{n \geq 1}$ obeys the following divisibility properties:

\begin{enumerate}

\item Any prime $p$ occurs infinitely often:
\[
  n \equiv -1 \pmod{p} \quad (n > p)
  \quad \implies \quad
  p \mid \pp_n.
\]

\item Arbitrarily large intervals of consecutive members exist such that
\[
  p \mid \pp_n
  \quad \implies \quad
  p \mid \pp_{n p^r+b} \quad (0 \leq b < p^r, \, r \geq 1).
\]

\item Arbitrarily many prime factors occur, in particular:
\[
  \rad^*(n+1) \mid \pp_n.
\]

\end{enumerate}
\end{theorem}

\begin{theorem}[Kellner and Sondow \cite{Kellner&Sondow:2017}] \label{thm:primes-bound}
If $n \geq 1$, then
\[
  \pp_n = \prod_{ \substack{
    p \, \leq \, \frac{n+1}{\lambda_n} \\
    s_p(n) \, \geq \, p}} p,
\]
where
\[
  \lambda_n := \left\{
    \begin{array}{ll}
      2, & \text{if $n$ is odd,} \\
      3, & \text{if $n$ is even.}
    \end{array}
  \right.
\]
In particular, the divisor $\lambda_n$ is best possible,
respectively the bound $\frac{n+1}{\lambda_n}$ is sharp,
for odd $n=2p-1$ and even $n=3p-1$, when $p$ is an odd prime.
\end{theorem}

The divisor $\lambda_n$ can be improved by accepting additional conditions.

\begin{theorem} \label{thm:primes-bound2}
The divisor $\lambda_n = 4$ holds in Theorem~\ref{thm:primes-bound},
\begin{enumerate}
  \item if $n = 4$,
  \item if $n \geq 10$ is even and $n \notin \bigcup_{p \geq 5} \set{3p-1, 4p-2}$,
  \item if $n \geq 11$ is odd and $n \notin \bigcup_{p \geq 5} \set{2p-1, 3p-2, 4p-3}$.
\end{enumerate}
\end{theorem}

The optimal divisor $\lambda_n^*$, which could replace $\lambda_n$ in
Theorem~\ref{thm:primes-bound}, obviously satisfies
\[
  \lambda_n \leq \lambda_n^* := \frac{n+1}{\GP(\pp_n)}.
\]

First values of $\pp_n$, $\GP(\pp_n)$, and $\lambda_n^*$ are given in
Table~\ref{tbl:primes-values}.
Supported by some computations, we raise the following conjecture,
which implies an upper bound on $\lambda_n^*$.

\begin{conjecture} \label{conj:lambda-bound}
We have the estimates that
\[
  \GP(\pp_n) > \sqrt{n}, \quad \lambda_n^* < \sqrt{n} \quad (n > 192).
\]
\end{conjecture}

We further introduce the decomposition
\[
  \pp_n = \pp_n^- \cdot \pp_n^+,
\]
where
\[
  \pp_n^- := \prod_{\substack{
    p \, < \, \sqrt{n} \\
    s_p(n) \, \geq \, p}} p
  \quad \text{and} \quad
  \pp_n^+ := \prod_{\substack{
    p \, > \, \sqrt{n} \\
    s_p(n) \, \geq \, p}} p.
\]

Note that the omitted condition $p = \sqrt{n}$ has no effect on the above decomposition.
Indeed, if $p = \sqrt{n} \in \ZZ$, then $p^2 = n$ and so $s_p(n) = 1$.
We keep in mind that the prime factors of $\pp_n^+$ are implicitly bounded by
Theorem~\ref{thm:primes-bound}.
Let $\intpart{\,\cdot\,}$ denote the integer part.
Define the additive function $\omega(n)$ counting the prime divisors of $n$.

\begin{theorem} \label{thm:primes-count}
If $n \geq 2$, then
\[
  \omega( \pp_n^+ ) = \sum_{\substack{
    p \, > \, \sqrt{n}\\
    \intpart{\frac{n-1}{p-1}} \, > \, \intpart{\frac{n}{p}}}} \!\! 1.
\]
Moreover, we have the estimates
\[
  \omega( \pp_n^- ) < \frac{5}{2} \frac{\sqrt{n}}{\log n}
  \quad \text{and} \quad
  \omega( \pp_n^+ ) < \sqrt{n}.
\]
\end{theorem}

The estimate of $\omega( \pp_n^+ )$ is apparently better than counting primes
in the interval $\bigl[ \sqrt{n}, \frac{n+1}{\lambda_n} \bigr]$
by the  prime-counting function $\pi(x) \sim \frac{x}{\log x}$,
while the obvious estimate $\omega( \pp_n^- ) \leq \pi(\sqrt{n})$ is sharp,
see Table~\ref{tbl:primes-values2}.

Conjecture~\ref{conj:lambda-bound} is equivalent to $\omega( \pp_n^+ ) > 0$
for $n > 192$.
On the basis of advanced computations, we raise the following conjecture,
which gives even more evidence to hold with a better approximation.

\begin{conjecture} \label{conj:primes-asymp}
There exists a constant $\kappa > 1$ such that
\[
  \omega( \pp_n^+ ) \, \sim \, \kappa \, \frac{\sqrt{n}}{\log n}
  \quad \text{as $n \to \infty$.}
\]
\end{conjecture}

Computations up to $n = 10^7$ suggest the value $\kappa = 1.8$ in that range
and an error term $O(\log n)$, see Figure~\ref{fig:graph1}.
Conjecture~\ref{conj:primes-asymp} implies at once the much weaker
Conjecture~\ref{conj:lambda-bound} for sufficiently large values.
However, both conjectures remain open.

Note that Theorem~\ref{thm:primes-bound}, as well as Theorem~\ref{thm:primes-denom}
in the next section, were recently given by Sondow and the author in
\cite[Thm.~1,~2,~4]{Kellner&Sondow:2017} using other notations.
We will choose here a quite different approach, starting from the more general
product~\eqref{eq:primes-prod}, to attain to Theorems~\ref{thm:primes-bound}
and~\ref{thm:primes-denom} by means of $p$-adic methods,
which result in short and essentially different proofs.

Outline for the rest of the paper: The next section shows the relations between
$\pp_n$ and the denominators of the Bernoulli polynomials in
Theorem~\ref{thm:primes-denom}.
Section~\ref{sec:div-prop} demonstrates the divisibility properties of $\pp_n$
and contains the proofs of Theorems~\ref{thm:primes-div} -- \ref{thm:primes-bound2}.
In Section~\ref{sec:step-func} we use step functions on a hyperbola to give a
proof of Theorem~\ref{thm:primes-count}.
Section~\ref{sec:denom-poly} discusses the $p$-adic valuation of polynomials and
includes the proof of Theorem~\ref{thm:primes-denom}.


\section{Bernoulli polynomials} \label{sec:bern-poly}

The Bernoulli polynomials are defined by the generating function
\[
  \frac{t e^{xt}}{e^t - 1} = \sum_{n \geq 0} B_n(x) \frac{t^n}{n!} \quad (|t| < 2\pi)
\]
and are explicitly given by
\begin{equation} \label{eq:bern-poly}
  B_n(x) = \sum_{k=0}^{n} \binom{n}{k} B_k \, x^{n-k} \quad (n \geq 0),
\end{equation}
where $B_k = B_k(0)$ is the $k$th Bernoulli number.
First values are
\begin{equation} \label{eq:bern-value}
  B_0 = 1, \quad B_1 = -\frac{1}{2}, \quad B_2 = \frac{1}{6},
  \quad B_4 = -\frac{1}{30}, \quad B_6 = \frac{1}{42},
\end{equation}
while $B_k = 0$ for odd $k > 1$, see~\cite[Chap.~3.5, pp.~112--125]{Prasolov:2010}.
The von~Staudt--Clausen theorem describes the denominator of the
Bernoulli number with even index.
Together with \eqref{eq:bern-value}, we have the squarefree denominators
\begin{equation} \label{eq:bern-denom}
  \dd_n := \denom( B_n ) =
    \left\{
      \begin{array}{ll}
        2, & \text{if $n = 1$}, \\
        1, & \text{if $n \geq 3$ is odd}, \\
        \prod\limits_{p-1 \, \mid \, n} p, & \text{if $n \geq 2$ is even}
      \end{array}
    \right.
  \quad (n \geq 1).
\end{equation}
In addition, we define the related polynomials
\begin{equation} \label{eq:def-bt}
  \BT_n(x) := B_n(x) - B_n,
\end{equation}
which have no constant term. Considering the power-sum function
\[
  S_n(x) := 1^n + 2^n + \dots + (x-1)^n \quad (x \in \NN),
\]
it is well known that
\begin{equation} \label{eq:power-sum}
  S_n(x) = \frac{\BT_{n+1}(x)}{n+1},
\end{equation}
implying that $\BT_n(x)$ is an integer-valued function.
The denominators of the polynomials $\BT_n(x)$, $B_n(x)$, and $S_n(x)$ are
surprisingly connected with the product~\eqref{eq:primes-prod} as follows.

\begin{theorem}[Kellner and Sondow \cite{Kellner&Sondow:2017}] \label{thm:primes-denom}
If $n \geq 1$, then we have the relations
\begin{alignat*}{3}
  &\denom( \BT_n(x) ) &&= \pp_n, \\
  &\denom( B_n(x) )   &&= \lcm( \pp_n, \dd_n ), \\
  &\denom( S_n(x) )   &&= (n+1) \, \pp_{n+1}.
\end{alignat*}
\end{theorem}

For an explicit product formula of $\denom( B_n(x) )$,
we refer to \cite[Thm.~4]{Kellner&Sondow:2017}.


\section{Divisibility properties} \label{sec:div-prop}

An integer $n \geq 0$ has a unique finite $p$-adic expansion
$n = \sum\limits_{k=0}^{r} a_k \, p^k$
with a definite $r \geq 0$ and base-$p$ digits $a_k$ satisfying
$0 \leq a_k \leq p-1$.
The sum of these digits defines the function
$s_p( n ) := \sum\limits_{k=0}^{r} a_k$.
Note that
\begin{equation} \label{eq:sp-less-p}
  s_p(n) < p  \quad (p \geq n),
\end{equation}
since $s_p(n) = 1$ if $p = n$, and $s_p(n) = n$ if $p > n$.
Moreover, we have some complementary results as follows.

\begin{lemma} \label{lem:sp-all-n}
If $n \geq 1$ and $p$ is a prime, then
\[
  \frac{n+1}{2} < p < n \quad \implies \quad s_p(n) < p.
\]
\end{lemma}

\begin{proof}
By assumption we have $p < n < 2p-1$. Thus, we can write $n = a_0 + p$
with $0 < a_0 < p-1$. This implies $s_p(n) = a_0 + 1 < p$.
\end{proof}

\begin{lemma} \label{lem:sp-even-n}
If $n \geq 2$ is even and $p$ is a prime, then
\[
  \frac{n+1}{3} < p < \frac{n+1}{2} \quad \implies \quad s_p(n) < p.
\]
\end{lemma}

\begin{proof}
Since $n$ is even, we infer that $2p \leq n < 3p-1$.
If $p=2$, then we only have the case $n=4$, so $s_p(n) = 1 < p$.
For odd $p \geq 3$ we obtain
$n = a_0 + 2p$ with $a_0 \leq p-3$.
Consequently, $s_p(n) = a_0 + 2 < p$.
\end{proof}

\begin{lemma} \label{lem:sp-lambda}
Let $n \geq 1$, $p$ be an odd prime, and $1 \leq \lambda < p$.
Write $n = a_0 + a_1 \, p + \dotsm$.
Then
\[
  \frac{n+1}{\lambda+1} < p < \frac{n+1}{\lambda}
  \quad \text{and} \quad a_0 < p - \lambda
  \quad \implies \quad s_p(n) < p.
\]
\end{lemma}

\begin{proof}
The left inequation above yields
\[
  \lambda \, p \leq n \leq p-2 + \lambda \, p.
\]
Since $\lambda < p$ and $a_0 < p - \lambda$,
we have $0 \leq a_0 < p - \lambda$,
implying that $s_p(n) = a_0 + \lambda < p$.
\end{proof}

\begin{proof}[Proof of Theorem~\ref{thm:primes-div}]
Let $n \geq 1$ and $p$ be a prime.
Recall the product of $\pp_n$ in \eqref{eq:primes-prod}.

(a) If $n > p$ and $n \equiv -1 \pmod{p}$,
then we can write
\[
  n = p-1 + m \, p
\]
with some $m \geq 1$. This implies $s_p(n) \geq p$ and so $p \mid \pp_n$.
Since this holds for all $m \geq 1$, the prime $p$
occurs infinitely often as a divisor in the sequence $(\pp_n)_{n \geq 1}$.

(b) If $p \mid \pp_n$, then $s_p(n) \geq p$. Let $r \geq 1$. For all $b$ with
$0 \leq b < p^r$, one observes that
\[
  m := b + n \, p^r
\]
yields no carries in its $p$-adic expansion.
Thus, we have $s_p(m) = s_p(b) + s_p(n) \geq p$. Since $r \geq 1$,
the interval $[0,p^r-1]$ can be arbitrarily large.

(c) Neglecting the trivial case, we assume that $n+1$ is composite and so
$\rad^*(n+1) > 1$.
For all prime divisors $p$ of $n+1$ we then infer by part (a) that $s_p(n) \geq p$.
This shows that
\[
  \rad^*(n+1) \mid \pp_n.
\]
Now we construct for $r \geq 2$ different primes $p_k$ an index $n$ such that
\[
  n = -1 + \prod_{k = 1}^{r} p_k
  \quad \implies \quad
  \prod_{k = 1}^{r} p_k \mid \pp_n,
\]
implying that arbitrarily many prime factors can occur.
\end{proof}

\begin{proof}[Proof of Theorem~\ref{thm:primes-bound}]
By Lemma~\ref{lem:sp-all-n} and \eqref{eq:sp-less-p} we deduce for $n \geq 1$ that
\begin{equation} \label{eq:primes-lambda}
  \pp_n = \prod_{ \substack{
    p \, \leq \, \frac{n+1}{\lambda'} \\
    s_p(n) \, \geq \, p}} p
\end{equation}
holds with $\lambda' = 2$. If $p \geq 3$, then $n=2p-1$ is odd and $s_p(n) = p$,
showing that the bound is sharp in this case. Now, let $n \geq 2$ be even.
Since $\frac{n+1}{2} \notin \ZZ$, we infer by using Lemma~\ref{lem:sp-even-n}
as a complement that \eqref{eq:primes-lambda} also holds with $\lambda' = 3$.
This defines $\lambda_n = 3$ for even $n$, while $\lambda_n = 2$ for odd $n$.
If $p \geq 3$, then $n=3p-1$ is even and $s_p(n) = p+1$, giving a sharp bound
for that case.
\end{proof}

\begin{proof}[Proof of Theorem~\ref{thm:primes-bound2}]
We have to determine the cases of $n$, where $\lambda_n = 4$ holds in
Theorem~\ref{thm:primes-bound},
or rather $\lambda' = 4$ holds in \eqref{eq:primes-lambda}.
The exceptional cases $n = 2p-1$ and $n = 3p-1$ for odd $p$ are already handled
by Theorem~\ref{thm:primes-bound}
providing the optimal values $\lambda_n = 2$ and $\lambda_n = 3$, respectively.

Regarding entries of $\lambda_n^*$ ($n \leq 9$) in Table~\ref{tbl:primes-values},
one observes that $\lambda_n = 4$ only holds for $n=4$ in this range.
This proves part (a).

Let $n \geq 10$.
Since $\lambda_n^* = 5.5$ for $n=10$ (see Table~\ref{tbl:primes-values}),
and $(n+1)/4 \geq 3$ for $n \geq 11$, the primes $p=2$ and $p=3$ are always
considered in \eqref{eq:primes-lambda}, when possibly taking $\lambda' = 4$.
From now let $p \geq 5$ be fixed. Write $n = a_0 + a_1 \, p + \dotsm$.
We distinguish between the following two cases.

Case $n \geq 10$ even:
We have $\lambda_n = 3$ by Theorem~\ref{thm:primes-bound}.
If $n \neq 3p-1$, we infer by Lemma~\ref{lem:sp-lambda} with $\lambda = 3$,
that $a_0 \leq p-4$ must hold. Since $n$ is even, the only exception can appear
by parity if $a_0 = p-2$, so $n = 4p-2$ and $s_p(n) = p+1$.
This defines the set of exceptions $\EE_p := \set{3p-1, 4p-2}$ in this case.

Case $n \geq 11$ odd:
We have $\lambda_n = 2$ by Theorem~\ref{thm:primes-bound}.
If $n \neq 2p-1$, then we deduce
from Lemma~\ref{lem:sp-lambda} with $\lambda = 2$, that $a_0 \leq p-3$ must hold.
Since $n$ is odd and due to parity, the only exception can happen,
when $a_0 = p-2$, so $n = 3p-2$ and $s_p(n) = p$.
If also $n \neq 3p-2$, then we derive from Lemma~\ref{lem:sp-lambda} with $\lambda = 3$,
that $a_0 \leq p-4$ must hold. Again, the only exception can occur
with $a_0 = p-3$, so $n = 4p-3$ and $s_p(n) = p$. This defines the set of exceptions
$\EO_p := \set{2p-1, 3p-2, 4p-3}$ in that case.

Consequently, if $n$ is even and $n \notin \EE_p$, respectively $n$ is odd
and $n \notin \EO_p$, for all $p \geq 5$,
then $\lambda_n = 4$. This proves parts (b) and (c), completing the proof.
\end{proof}


\section{Step functions} \label{sec:step-func}

As usual, we write $x = \intpart{x} + \fracpart{x}$, where $0 \leq \fracpart{x} < 1$
denotes the fractional part. We define for $n \geq 2$ the step functions,
giving the integer part of a hyperbola, by
\[
  \phi_n(x) := \intpart{\frac{n-1}{x-1}}
  \quad \text{and} \quad
  \psi_n(x) := \intpart{\frac{n}{x}},
\]
and their difference
\[
  \Delta_n(x) := \phi_n(x) - \psi_n(x)
\]
on the intervals $(\sqrt{n},\infty)$. Note that
\begin{align}
  \phi_n(x) = \psi_n(x) &=
    \left\{
      \begin{array}{ll}
        1, & \text{if $x = n$,} \\
        0, & \text{if $x > n$,}
      \end{array}
    \right. \label{eq:phi-psi} \\
\shortintertext{and}
  \Delta_n(x) &= 0 \quad (x \geq n). \label{eq:delta-0}
\end{align}

Since we are here interested in summing the function $\Delta_n(x)$,
it should be noted that there is a connection with Dirichlet's divisor problem.
This can be stated with Voronoi's error term as
\[
  \sum_{k=1}^{n} \, \intpart{\frac{n}{k}} = n \log n + (2\gamma-1) n
    + O \bigl( n^{\frac{1}{3}+\varepsilon} \bigr),
\]
where $\gamma = 0.5772\cdots$ is Euler's constant. The exponent~$\tfrac{1}{3}$
in the error term was gradually improved by several authors, currently to
$\frac{131}{416} = 0.3149\cdots$ by Huxley, see \cite[Chap.~10.2, p.~182]{Cohen:2007}.
Rather than using analytic theory, we will use a counting argument below.
Before proving Theorem~\ref{thm:primes-count}, we need some lemmas.

\begin{lemma} \label{lem:delta-1}
If $n \geq 2$, then
\[
  \Delta_n(x) \in \set{0,1} \quad (x > \sqrt{n}).
\]
\end{lemma}

\begin{proof}
Let $n \geq 2$ be fixed. By \eqref{eq:delta-0}
it remains to consider the range $n > x > \sqrt{n}$.
By using $\intpart{x} = x - \fracpart{x}$,
we easily infer that
\[
  \Delta_n(x) = \phi_n(x) - \psi_n(x)
    = \frac{\frac{n}{x}-1}{x-1} + \fracpart{\frac{n}{x}}
    - \fracpart{\frac{n-1}{x-1}} \in \set{0,1},
\]
because all summands of the right-hand side lie in the interval $[0,1)$.
\end{proof}

\begin{lemma} \label{lem:delta-2}
If $n \geq 2$ and $X > \sqrt{n}$ where $\phi_n(X) \leq 1$, then
\[
  \Delta_n(x) = 0 \quad (x \geq X).
\]
\end{lemma}

\begin{proof}
If $\phi_n(X) = 0$, we must have $X > n$, and we are done by \eqref{eq:phi-psi}
and \eqref{eq:delta-0}.
So it remains the case $\phi_n(X) = 1$ and $X \leq n$. By \eqref{eq:phi-psi}
we conclude that $\phi_n(x) = 1$ for all $x \in [X,n]$.
Since $\phi_n(x) \geq \psi_n(x)$ for $x \geq X$ by Lemma~\ref{lem:delta-1},
it also follows that $\psi_n(x) = 1$, and so $\Delta_n(x) = 0$, for all $x \in [X,n]$.
Together with \eqref{eq:delta-0} this gives the result.
\end{proof}

\begin{lemma} \label{lem:delta-sp}
If $n \geq 2$ and $p > \sqrt{n}$ is a prime, then
\[
  s_p(n) \geq p
  \quad \iff \quad
  \Delta_n(p) = 1.
\]
\end{lemma}

\begin{proof}
Let $n \geq 2$ be fixed.
If $p \geq n$, then we have $s_p(n) < p$ and $\Delta_n(p) = 0$
by \eqref{eq:sp-less-p} and \eqref{eq:delta-0}, respectively,
and we are done. It remains the range $n > p > \sqrt{n}$.
In view of Lemma~\ref{lem:delta-1}, we have to show that
\begin{equation} \label{eq:sp-intpart}
  s_p(n) \geq p
  \quad \iff \quad
  \intpart{\frac{n-1}{p-1}} > \intpart{\frac{n}{p}}
\end{equation}
for the prime $p$ in that range. Thus, we can write
\[
  n = a_0 + a_1 \, p \quad \text{and} \quad s_p(n) = a_0 + a_1,
\]
where $a_1 = \intpart{\dfrac{n}{p}} \geq 1$. Substituting the $p$-adic digits leads to
\[
  s_p(n) = n - (p-1) \intpart{\frac{n}{p}}.
\]
If $s_p(n) \geq p$, then we deduce the following steps:
\begin{align*}
  n - 1 - (p-1) \intpart{\frac{n}{p}} &\geq \, p - 1, \\
  \frac{n-1}{p-1} &\geq \intpart{\frac{n}{p}} + 1, \\
  \intpart{\frac{n-1}{p-1}} &> \intpart{\frac{n}{p}}.
\end{align*}
Since the statements above also hold in reverse order, \eqref{eq:sp-intpart} follows.
\end{proof}

\begin{lemma} \label{lem:delta-sum}
If $n \geq 2$, then
\[
  \sum_{k \, > \, \sqrt{n}} \Delta_n(k) < \sqrt{n}.
\]
\end{lemma}

\begin{proof}
Let $n \geq 2$ be fixed. Set $\IT_n := \bigl[ \intpart{\sqrt{n}\,}+1, n \bigr]$
and $\IT^*_n := \IT_n \, \cap \, \ZZ$, where both sets are not empty.
Considering Lemma~\ref{lem:delta-1} and \eqref{eq:delta-0},
we have to count the events when $\Delta_n(k) = 1$ for $k \in \IT^*_n$.
The images $\phi_n(\IT_n)$ and $\psi_n(\IT_n)$ describe both a graded hyperbola
(see Figure~\ref{fig:graph2}),
being piecewise constant and divided into decreasing steps.
From now on, we are interested in the properties of $\phi_n(x)$.
For $x \in \IT_n$  we call $h = \phi_n(x)$
the height of the corresponding step in the interval
$\IT_n(h) := \phi_n^{-1}(h) \cap \IT_n$.

Viewing the function $\phi_n(x)$ on the interval $\IT_n$,
we observe the steps of decreasing heights
\begin{equation} \label{eq:step-height}
  h = \ell, \ell-1, \dots, 2, 1,
\end{equation}
where the heights are bounded by the values of $\phi_n(x)$ on the boundary of
$\IT_n$, namely,
\begin{equation} \label{eq:step-ell}
  \ell := \phi_n\bigl( \intpart{\sqrt{n}\,}+1 \bigr)
    = \intpart{\frac{n-1}{\intpart{\sqrt{n}\,}}} \geq 1
  \quad \text{and} \quad
  \phi_n(n) = 1.
\end{equation}
Hence, we have a decomposition of $\IT_n$ into the disjoint intervals
$\IT_n(1), \dots, \IT_n(\ell)$.

Now fix a height $h \in \{1,\dots,\ell\}$.
Set $\IT^*_n(h) := \IT_n(h) \, \cap \, \ZZ$.
It turns out that on the interval $\IT_n(h)$
the event $\Delta_n(k) = 1$ can at most happen once.
More precisely, $k \in \IT^*_n(h)$ must be the greatest possible integer
(see gray areas in Figure~\ref{fig:graph2}).
Assume to the contrary that there exist integers $k, k' \in \IT^*_n(h)$
satisfying $k < k'$ and $\Delta_n(k) = 1$. By definition we have
$h = \phi_n(k) = \phi_n(k')$.
Thus, it also follows, since $\phi_n(x)$ is constant on the interval $\IT_n(h)$,
that $h = \phi_n(k) = \phi_n(k+1)$, where $k+1 \leq k'$.
Putting all together, we then obtain that
\[
  \underbrace{\intpart{\frac{n-1}{k}}}_{\phi_n(k+1)}
    = \underbrace{\intpart{\frac{n-1}{k-1}}
    > \intpart{\biggl.\frac{n}{k}\biggr.}}_{\Delta_n(k) = 1}
    = \intpart{\frac{n-1}{k} + \frac{1}{k}},
\]
giving a contradiction.
\pagebreak

As a consequence, we have now to count the intervals $\IT_n(h)$ or rather the
steps of different heights $h$.
In total, there are $\ell$~such ones by \eqref{eq:step-height}.
Next we show that the step of height $h=1$ has to be excluded from counting.
Indeed, this follows by Lemma~\ref{lem:delta-2}, since for any
$k \in \IT^*_n$ with $\phi_n(k) \leq 1$, we always have $\Delta_n(k) = 0$.

We finally deduce that
\[
  \sum_{k \, > \, \sqrt{n}} \Delta_n(k)
  = \sum_{k \, \in \, \IT^*_n} \Delta_n(k) \leq \ell - 1.
\]
It remains to show that $\ell - 1 < \sqrt{n}$. By \eqref{eq:step-ell} this turns into
\[
  \intpart{\frac{n-1}{\intpart{\sqrt{n}\,}}} < \sqrt{n} + 1,
\]
which holds by the stricter inequality
\[
  n < \intpart{\sqrt{n}} (\sqrt{n} + 1) + 1
  = n + (\sqrt{n} + 1)(\underbrace{1 - \fracpart{\sqrt{n}}}_{> \, 0}),
  \vspace*{-1em}
\]
implying the result.
\end{proof}

\begin{proof}[Proof of Theorem~\ref{thm:primes-count}]
Let $n \geq 2$. First we show that
\begin{equation} \label{eq:primes-sqrt}
  \omega( \pp_n^+ ) = \sum_{\substack{p \, > \, \sqrt{n}\\
    \intpart{\frac{n-1}{p-1}} \, > \, \intpart{\frac{n}{p}}}} \!\!\! 1
    < \sqrt{n}.
\end{equation}
By combining Lemmas~\ref{lem:delta-1}, \ref{lem:delta-sp}, and~\ref{lem:delta-sum},
we deduce that
\[
  \omega( \pp_n^+ )
  = \sum_{\substack{p \, > \, \sqrt{n}\\
    s_p(n) \, \geq \, p}} \! 1
  = \sum_{\substack{p \, > \, \sqrt{n}\\
    \Delta_n(p) \, = \, 1}} \! 1
  \leq \sum_{k \, > \, \sqrt{n}} \Delta_n(k)
  < \sqrt{n}.
\]
Rewriting the condition $\Delta_n(p) = 1$ as in \eqref{eq:sp-intpart} finally
yields \eqref{eq:primes-sqrt}.

Next we use the straightforward estimate
\[
  \omega( \pp_n^- ) \leq \pi(\sqrt{n}).
\]
By \cite[Cor.~2, p.~69]{Rosser&Schoenfeld:1962} we have
\[
  \pi(x) < \frac{5}{4} \frac{x}{\log x} \quad (x > 1, \, \intpart{x} \neq 113)
\]
with exceptions at $\intpart{x} = 113$. More precisely, we have $\pi(113) = 30$,
while $\frac{5}{4}\frac{113}{\log 113} = 29.8\cdots$.
Since $\omega( \pp_{113^2}^- ) = 19$, we infer that
\[
  \omega( \pp_n^- ) < \frac{5}{2} \frac{\sqrt{n}}{\log n}
\]
holds for all $n \geq 2$, finishing the proof.
\end{proof}


\section{\texorpdfstring{$p$-adic}{p-adic} valuation of polynomials} \label{sec:denom-poly}

Let $\QQ_p$ be the field of $p$-adic numbers. Define $\ord_p(s)$ as the
$p$-adic valuation of $s \in \QQ_p$.
The ultrametric absolute value $\norm{\cdot}_p$ is defined by
$\norm{s}_p := p^{-\ord_p(s)}$ on $\QQ_p$.
Let $\norm{\cdot}_\infty$ be the usual norm on $\QQ_\infty = \RR$.

These definitions can be uniquely extended to a nonzero polynomial
\[
  f(x) = \sum_{k=0}^{r} c_k \, x^k \quad (f \in \QQ[x] \backslash \set{0}),
\]
where $r = \deg f$. We explicitly omit the case $f=0$ for simplicity in the following.
Define
\begin{align}
  \ord_p(f)  &:= \min_k \ord_p( c_k ), \label{eq:ordp-func} \\
  \norm{f}_p &:= p^{-\ord_p(f)} = \max_k \norm{c_k}_p \!, \nonumber \\
  \norm{f}_\infty &:= \cont(f), \nonumber
\end{align}
where $\cont(\cdot)$ gives the unsigned content of a polynomial,
see \cite[Chap.~5.2, p.~233]{Robert:2000} and \cite[Chap.~2.1, p.~49]{Prasolov:2010}.
The product formula then states that
\[
  \prod_{p \, \in \, \PP \, \cup \, \set{\infty}} \!\!\norm{f}_p = 1,
\]
including the classical case $f \in \QQ^\times$ as well.
It also follows by definition that
\begin{equation} \label{eq:cont-prod}
  \cont(f) = \prod_{p \, \in \, \PP} p^{\ord_p(f)}.
\end{equation}

Before giving the proof of Theorem~\ref{thm:primes-denom}, we have two lemmas.
To avoid ambiguity, e.g., between $B_n(x)$ and $B_n$, we explicitly write
$f(x)$ instead of $f$ below.

\begin{lemma}[Carlitz~\cite{Carlitz:1961}] \label{lem:Carlitz}
If $d, \ell, n$ are integers and $p$ is a prime such that $0 < d < n$
and $d \mid p-1$, then
\[
  p \mid \binom{n}{\ell d} \quad (0 < \ell d < n) \quad \iff \quad s_p(n) \leq d.
\]
\end{lemma}

\begin{lemma} \label{lem:ordp-bt}
If $n \geq 1$ and $p$ is a prime, then
\[
  \ord_p (\BT_n(x)) =
    \left\{
      \begin{array}{rl}
        -1, & \text{if $s_p(n) \geq p$}, \\
        0,  & \text{else}.
      \end{array}
    \right.
\]
\end{lemma}

\begin{proof}
We initially compute $\BT_n(x)$ by \eqref{eq:bern-poly}, \eqref{eq:bern-value},
and~\eqref{eq:def-bt}.

Cases $n=1,2$: We obtain $\BT_1(x) = x$ and $\BT_2(x) = x^2-x$.
Thus, we have $\ord_p(\BT_n(x)) = 0$, while $s_p(n) < p$, for all primes $p$;
showing the result for these cases.

Now let $n \geq 3$. Since $B_k = 0$ for odd $k \geq 3$, we deduce that
\[
  \BT_n(x) = x^n - \frac{n}{2} x^{n-1} +
    \sum_{\substack{k=2\\ 2 \, \mid \, k}}^{n-1} \binom{n}{k} B_k \, x^{n-k}.
\]
Evaluating the coefficients of $\BT_n(x)$ by \eqref{eq:ordp-func},
we show that
\begin{equation} \label{eq:ordp-bt-values}
  \ord_p (\BT_n(x)) \in \set{-1, 0}.
\end{equation}
On the one hand, $\BT_n(x)$ is a monic polynomial implying that
\[
  \ord_p (\BT_n(x)) \leq 0.
\]
On the other hand, we derive that
\begin{equation} \label{eq:ordp-coeff-1}
  e_1 := \ord_p\left( \frac{n}{2} \right) \geq
    \left\{
      \begin{array}{rl}
        -1, & \text{if $p = 2$}, \\
         0, & \text{if $p \geq 3$},
      \end{array}
    \right.
\end{equation}
and for even $k$ with $2 \leq k < n$ that
\begin{equation} \label{eq:ordp-coeff-k}
  e_k := \ord_p \left( \binom{n}{k} B_k \right) \geq
    \left\{
      \begin{array}{rl}
        -1, & \text{if $p-1 \mid k$}, \\
         0, & \text{else},
      \end{array}
    \right.
\end{equation}
since the von Staudt--Clausen theorem in \eqref{eq:bern-denom} reads
\begin{equation} \label{eq:ordp-bern}
  \ord_p ( B_k ) =
    \left\{
      \begin{array}{rl}
        -1,     & \text{if\ } p-1 \mid k, \\
        \geq 0, & \text{else.}
      \end{array}
    \right.
\end{equation}
This all confirms \eqref{eq:ordp-bt-values}. Next we consider the cases
$p \geq 3$ and $p=2$ separately.

Case $p \geq 3$: Since $e_1 \geq 0$ by \eqref{eq:ordp-coeff-1},
it remains to evaluate \eqref{eq:ordp-coeff-k}.
Set $d = p-1 \geq 2$. In view of \eqref{eq:ordp-coeff-1} -- \eqref{eq:ordp-bern},
we use Lemma~\ref{lem:Carlitz} to establish that
\begin{alignat*}{3}
  \ord_p (\BT_n(x)) = 0 \quad &\iff \quad e_{\ell d} \geq 0 \quad & (0 < \ell d < n) \\
    &\iff \quad \ord_p \left( \binom{n}{\ell d} \right) \geq 1 \quad & (0 < \ell d < n) \\
    &\iff \quad s_p(n) \leq d.
\end{alignat*}
Together with \eqref{eq:ordp-bt-values}, this conversely implies that
\begin{equation} \label{eq:ordp-bt-1}
  \ord_p (\BT_n(x)) = -1 \quad \iff \quad s_p(n) \geq p,
\end{equation}
showing the result for $p \geq 3$.

\mbox{Case $p = 2$, $n \geq 3$ odd}:
We have $e_1 = -1$ by \eqref{eq:ordp-coeff-1},
and $s_p(n) \geq p$, since $n \geq 3$.
Thus \eqref{eq:ordp-bt-1} holds for this case.

\mbox{Case $p = 2$, $n \geq 4$ even}:
Since $e_1 \geq 0$ by \eqref{eq:ordp-coeff-1},
we have to evaluate \eqref{eq:ordp-coeff-k} and \eqref{eq:ordp-bern} once again.
In order to apply Lemma~\ref{lem:Carlitz} in the case $p=2$ with $d = p-1 = 1$,
we have to modify some arguments.
Note that $n = \binom{n}{1}$ is even.
Furthermore, if $\binom{n}{2\ell}$ is even, so is $\binom{n}{2\ell+1}$, since
\[
  \binom{n}{2\ell+1} \equiv \binom{n}{2\ell} \frac{n-2\ell}{2\ell+1} \equiv 0 \pmod{2}.
\]
Under the above assumptions we then infer that
\begin{alignat*}{4}
  \ord_p (\BT_n(x)) = 0 \quad &\iff \quad e_{2\ell} \geq 0 \quad && (0 < 2\ell < n) \\
    &\iff \quad \ord_p \left( \binom{n}{2\ell} \right) \geq 1 \quad && (0 < 2\ell < n) \\
    &\iff \quad \ord_p \left( \binom{n}{\ell} \right) \geq 1 \quad && (0 < \ell < n) \\
    &\iff \quad s_p(n) \leq d.
\end{alignat*}
This finally implies \eqref{eq:ordp-bt-1} and the result in that case;
completing the proof.
\end{proof}

\begin{proof}[Proof of Theorem~\ref{thm:primes-denom}]
Using the product formula~\eqref{eq:cont-prod}, we derive from
Lemma~\ref{lem:ordp-bt} that
\begin{equation} \label{eq:cont-prod-bt}
  \cont(\BT_n(x))^{-1}
  = \prod_{p \, \in \, \PP} p^{-\ord_p(\BT_n(x))}
  = \prod_{\substack{
    p \, \in \, \PP \\
    s_p(n) \, \geq \, p}} p.
\end{equation}
Hence, $\cont(\BT_n(x))^{-1}$ is a squarefree integer, giving the denominator,
and by \eqref{eq:primes-prod} we obtain
\[
  \denom(\BT_n(x)) = \cont(\BT_n(x))^{-1} = \pp_n.
\]
Furthermore, we deduce from \eqref{eq:bern-denom} and \eqref{eq:def-bt} that
\[
  \denom(B_n(x)) = \denom(\BT_n(x) + B_n) = \lcm(\pp_n, \dd_n).
\]
Finally, it follows by \eqref{eq:power-sum} that
$\cont(S_n(x)) = \cont(\BT_{n+1}(x)) / (n+1)$, and consequently that
\[
  \denom(S_n(x)) = (n+1) \, \pp_{n+1}. \qedhere
\]
\end{proof}


\section*{Conclusion}

As a result of Lemma~\ref{lem:ordp-bt} and \eqref{eq:cont-prod-bt},
the product~\eqref{eq:primes-prod} of $\pp_n$ is causally induced by the product
formula and arises from the $p$-adic valuation of $\BT_n(x)$.
The bounds given in Theorem~\ref{thm:primes-bound} are self-induced by
properties of $s_p(n)$ as shown by Lemmas~\ref{lem:sp-all-n} and~\ref{lem:sp-even-n}.


\appendix

\setcounter{Table}{0}
\renewcommand{\theTable}{A\arabic{Table}}
\renewcommand{\thefigure}{B\arabic{figure}}


\section{Tables}

\begin{Table} \label{tbl:primes-values}
First values of $\pp_n$ are small, while subsequent values are volatile for
larger indices ($\lambda_n^*$ rounded to $2$ decimal places):

\begin{center}
\begin{tabular}{|*{15}{c|}}
  \hline
  $n$           & 1 & 2 & 3 & 4 & 5 & 6 & 7 & 8 & 9 & 10 & 11 & 12 & 13 & 14 \\ \hline
  $\pp_n$       & 1 & 1 & 2 & 1 & 6 & 2 & 6 & 3 & 10 & 2 & 6 & 2 & 210 & 30 \\ \hline
  $\GP(\pp_n)$  & 1 & 1 & 2 & 1 & 3 & 2 & 3 & 3 & 5 & 2 & 3 & 2 & 7 & 5 \\ \hline
  $\lambda_n^*$ & 2 & 3 & 2 & 5 & 2 & 3.5 & 2.67 & 3 & 2 & 5.5 & 4 & 6.5 & 2 & 3 \\
  \hline
\end{tabular}
\end{center}
\smallskip
\begin{center}
\begin{tabular}{|*{9}{c|}}
  \hline
  $n$           & 100 & 101 & 102 & 103 & 104 & 105 & 106 & 107 \\ \hline
  $\pp_n$       & 1\,326 & 72\,930 & 4\,290 & 30\,030 & 2\,310 & 17\,490 & 330 & 330 \\ \hline
  $\GP(\pp_n)$  & 17 & 17 & 13 & 13 & 11 & 53 & 11 & 11 \\ \hline
  $\lambda_n^*$ & 5.94 & 6 & 7.92 & 8 & 9.55 & 2 & 9.73 & 9.82 \\
  \hline
\end{tabular}
\end{center}
\end{Table}

\medskip

\begin{Table} \label{tbl:primes-values2}
Number of prime factors of $\pp_n = \pp_n^- \cdot \pp_n^+$
compared to $\pi_n = \pi_n^- + \pi_n^+$, where
$\pi_n := \pi\bigl(\frac{n+1}{\lambda_n}\bigr)$
and $\pi_n^- := \pi\bigl(\sqrt{n}\bigr)$:
\begin{center}
\begin{tabular}{|*{11}{c|}}
  \hline
  $n$ & 200 & 201 & 202 & 203 & 204 & 205 & 206 & 207 & 208 & 209 \\ \hline
  $\GP(\pp_n)$ & 67 & 101 & 41 & 41 & 41 & 103 & 23 & 19 & 19 & 53 \\ \hline
  $\omega(\pp_n)$   & 8 & 11 & 10 & 8 & 7 & 9 & 8 & 7 & 6 & 6 \\ \hline
  $\omega(\pp_n^-)$ & 3 & 5 & 5 & 4 & 4 & 6 & 6 & 6 & 5 & 4 \\ \hline
  $\omega(\pp_n^+)$ & 5 & 6 & 5 & 4 & 3 & 3 & 2 & 1 & 1 & 2 \\ \hline
  $\pi_n$   & 19 & 26 & 19 & 26 & 19 & 27 & 19 & 27 & 19 & 27 \\ \hline
  $\pi_n^-$ & 6 & 6 & 6 & 6 & 6 & 6 & 6 & 6 & 6 & 6 \\ \hline
  $\pi_n^+$ & 13 & 20 & 13 & 20 & 13 & 21 & 13 & 21 & 13 & 21 \\
  \hline
\end{tabular}
\end{center}
\end{Table}

\vfill


\section{Figures}

\begin{figure}[htbp]
\begin{center}
\caption{Graph of $\omega( \pp_n^+ )$.}
\vspace*{1ex}
\includegraphics[width=12cm]{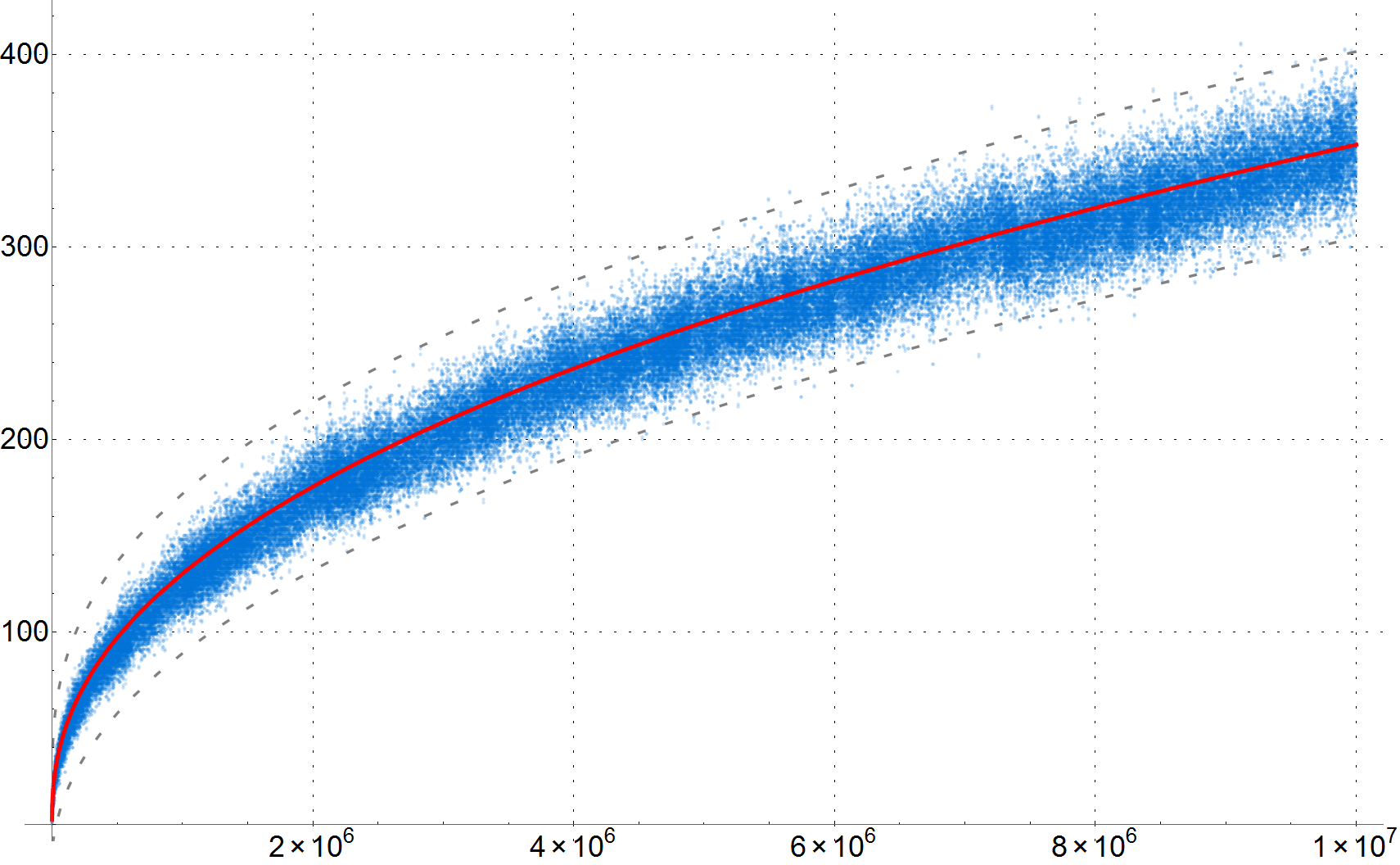}

\begin{minipage}{9cm}
\scriptsize
The red line displays the graph of $\kappa \, \frac{\sqrt{n}}{\log n}$ plotted
with $\kappa=1.8$, the dashed gray trend lines are plotted with offsets
$\pm 3 \log n$.
\end{minipage}
\label{fig:graph1}
\end{center}
\end{figure}

To compute the graph of Figure~\ref{fig:graph1} in a realizable time
and due to the limited display resolution, successive values of $n$ up to $10^7$
were chosen by random step sizes in the range $[100,200]$.
Incorporating the different bounds according to odd and even arguments in
Theorem~\ref{thm:primes-bound}, we computed both values of $\omega( \pp_n^+ )$
and $\omega( \pp_{n+1}^+ )$ for each chosen $n$. To illustrate the frequency,
the values were plotted by blue dots with an opacity of $20\%$.
All computations were performed by \textsl{Mathematica}.

\vfill
\pagebreak

\begin{figure}[htbp]
\begin{center}
\caption{Graph of $\phi_n(x)$ and $\psi_n(x)$ for $n=98$.}
\vspace*{1ex}
\includegraphics[width=12cm]{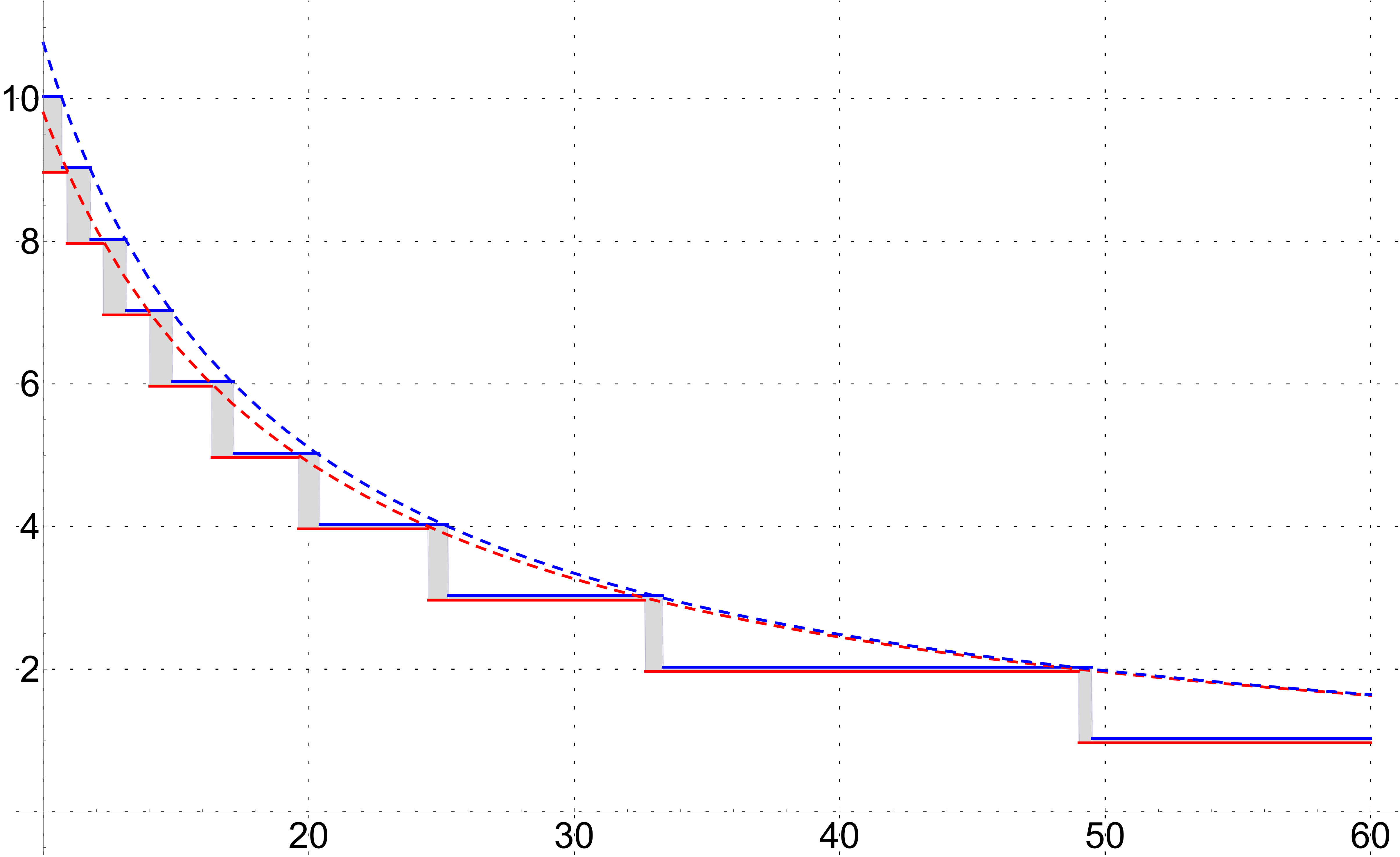}

\vspace*{1ex}
\begin{minipage}{11.8cm}
\scriptsize
The dashed lines display the graphs of the hyperbolas $\frac{n-1}{x-1}$ and
$\frac{n}{x}$ in blue and red, respectively.
The solid lines display the graphs of the step functions
$\phi_n(x) = \bigl[ \frac{n-1}{x-1} \bigr]$
and $\psi_n(x) = \bigl[ \frac{n}{x} \bigr]$ in blue and red
(with an extra spacing between them), respectively.
The gray areas indicate when the difference $\Delta_n(x) = 1$,
otherwise $\Delta_n(x) = 0$.
\end{minipage}
\label{fig:graph2}
\end{center}
\end{figure}


\section*{Acknowledgment}

The author is grateful for the valuable suggestions of the referee
which improved the paper.


\end{document}